\documentclass{article}

\usepackage{a4wide}

\usepackage{graphics,graphicx}
\usepackage{amsmath,amsthm}
\usepackage{amsfonts,mathrsfs,amssymb}

\newcommand{\D}{\ensuremath{{\cal D}}}

\newcommand{\E}{\ensuremath{{\cal E}}}

\newcommand{\mb}[1]{\ensuremath{\mathbb{#1}}}
\newcommand{\N}{\mb{N}}

\newcommand{\R}{\mb{R}}
\newcommand{\C}{\mb{C}}

\newcommand{\G}{\ensuremath{{\cal G}}}

\newcommand{\Char}{\ensuremath{\text{Char}}}

\renewcommand{\d}{\ensuremath{\partial}}
\newcommand{\diff}[1]{\frac{d}{d#1}}

\newfont{\bl}{msbm10 scaled \magstep2}

\newtheorem{theorem}{Theorem}[section]

\newtheorem{proposition}[theorem]{Proposition}

\theoremstyle{definition}
\newtheorem{remark}[theorem]{Remark}

\newcommand{\beq}{\begin{equation}}
\newcommand{\eeq}{\end{equation}}
\renewenvironment{proof}[1][Proof]{\textbf{#1.} }{\ \rule{0.5em}{0.5em}}

\newcommand{\dis}[2]{\langle #1 , #2 \rangle}
\newcommand{\notmid}{\mid\kern-0.5em\not\kern0.5em}

\newcommand{\al}{\alpha}
\newcommand{\be}{\beta}

\newcommand{\del}{\delta}
\newcommand{\eps}{\varepsilon}

\newcommand{\la}{\lambda}

\newcommand{\om}{\omega}

\renewcommand{\Re}{\ensuremath{\text{Re}}}

\newcommand{\ovl}[1]{\overline{#1}}

\title{Generalized solutions for the Euler-Bernoulli
 model with distributional forces\thanks{Supported by Ministry of
Science of Serbia, project 144016, and the Austrian
Science Fund (FWF) START program Y237 on 'Nonlinear distributional
geometry'}}

\author{\emph{G\"unther H\"ormann}\\
Fakult\"{a}t f\"{u}r Mathematik der Universit\"{a}t Wien\\
Nordbergstra{\ss}e 15, A-1090 Wien, Austria\\
\ \\
and\\
\ \\
\emph{Ljubica Oparnica}\\
Institute of Mathematics of the Serbian Academy of Science\\
Kneza Mihajla 35, 11000 Belgrade, Serbia}

\begin{document}

\maketitle

 \begin{abstract}
We establish existence and uniqueness of generalized solutions to the initial-boundary value problem corresponding to an Euler-Bernoulli beam model from mechanics. The governing partial differential equation is of order four and involves discontinuous, and even distributional, coefficients and right-hand side. The general problem is solved by application of functional analytic techniques to obtain estimates for the  solutions to regularized problems. Finally, we prove coherence properties and provide a regularity analysis of the generalized solution.\\[2mm]
 Key words: generalized solutions to partial differential equations, functional analytic methods,
 differential equations with discontinuous
 coefficients,  Colombeau generalized functions, nonlinear theories of generalized functions.\\[2mm]
 AMS 2000 Subject Classification: 35D05; 35D10, 46F30,  35Q72.
 \end{abstract}

 \section{Introduction and basic estimates}\label{intro}

  We consider the question of existence and uniqueness of a generalized solution to the initial-boundary value problem
   \begin{align}\label{P}
   \d^2_tu + Q(t,x,\d_x)u &= g, \nonumber \\
  u|_{t=0} &= f_1, \d_t u|_{t=0} = f_2,\\
  u|_{x=0} &=u|_{x=1}=0, \d_x u|_{x=0} =  \d_x u|_{x=1}=0,\nonumber
 \end{align}
  where $Q$ is a
  differential operator of the form
$$
  Q u  :=\d^2_x(c(x)\d^2_x)u + b(x,t)\d_x^2 u,
$$
  and $b,c,g,f_1$ and $f_2$ are generalized functions. The precise structure of the above problem is motivated by a model from mechanics describing the displacement of a beam under axial and transversal forces, which we briefly present in Subsection 1.1.

In Subsection 1.2 we review basics from Colombeau's nonlinear theory of generalized functions, which are the framework of our main results in Section 3.

In Subsection 1.3 we briefly sketch the well-known functional analytic tools from the variational approach and give a refined version of energy estimates, which lie at the heart of the proofs in Sections 2 and 3.

Section 2 is devoted to existence and uniqueness of solutions  in the Sobolev space setting for (slightly) regularized coefficients $b$ and $c$.

Finally, Section 3 establishes existence and uniqueness of a solution to the original problem in the space of Colombeau functions based on  families of weak solutions to the corresponding regularized problem. Furthermore, we prove coherence of the generalized solution with classical and weak solutions in case of smooth coefficients but also analyze generalized regularity properties in the non-smooth case.


  \subsection{The Euler-Bernoulli model}\label{model}

This subsection serves as a motivation
  for the analysis of problems of the type (\ref{P}).
  Consider the Euler-Bernoulli beam  under a distributed
  transversal force $g$ and an axial force $P$ (cf.\ \cite{Atanackovic:97}).   As has been initiated in \cite{HO:07} we pursue to investigate the case when the beam
  consists of two parts with
  different physical properties and cross sections.
Let $u$ denote the displacement depending on the spatial variable $x$ and time $t$.
Accordingly the differential equation
  of the transversal motion reads
\begin{equation}\label{PDE}
   \frac{\d^2}{\d x^2}\left( A(x)\frac{\d^2u}{d x^2}\right)
   +
    P(t)\frac{\d^2u}{\d x^2} + R(x)\frac{\d^2u}{\d t^2} =
    g(x,t),\quad x\in [0,1], \, t > 0,
\end{equation}
  where:
\begin{itemize}
   \item $A$ denotes the bending stiffness and is given by
    $A(x) = EI_1 + H(x-x_0)EI_2-EI_1$. Here, the constant
    $E$ is the modulus of elasticity and
    $I_1$, $I_2$ ($I_1\neq I_2$) are the moments of
    inertia that
    correspond to the two parts of the beam
    and $H$ is the Heaviside jump function.
  \item $R$ denotes the line density (i.e., mass per length) of the material and is of the form
    $ R(x)= R_0 + H(x-x_0)(R_1-R_2).$
  \item $P$ represents the axial force, given by a time dependent function or
  distribution. For example, it may be of the form $ P = P_0 + P_1 \delta (t-t_0)$ where $P_0>0$, $P_1>0$ and
  $t_0 > 0$, or $P = P_0 + P_1\sin(\om t)$ with $P_0>0$,
  $P_1>0$ and $\om >0$.
  \item $g$ is the force term and can be of the  form $g(x,t)= F_0\del (x-x_1)$,
  where $x_1\in (0,1)$. Furthermore, for applications also a specific dependence of $g$ on $u$ is of interest,
  as, for example, in the case of the so-called Winkler foundation with  $g = -c\, u$, $c > 0$ constant.
  \item As mentioned above, $u=u(x,t)$ denotes the displacement. Thus,
  its second derivatives $\frac{\d^2u}{\d t^2}$ and $\frac{\d^2u}{\d x^2}$
  represent acceleration and linearized
  curvature\footnote{without linearization it would read
  $k = \frac{\frac{\d^2u}{\d x^2}}{(1 + (\frac{\d u}{\d
  x})^2)^{\frac{3}{2}}}$.
  }, respectively.
  \end{itemize}
  In addition to  (\ref{PDE}) we consider initial conditions:
  \begin{equation}\label{ICpde}
  u(x,0) = f_1(x),\quad \d_t u (x,0) = f_2(x),
  \end{equation}
  where $f_1$ and $f_2$ are the initial displacement and the initial
  velocity.
  We consider the beam to be
  fixed at both ends and also supply
  boundary conditions of the form
  \begin{equation}\label{BCpde}
   u(0,t) = u(1,t) = 0, \quad \d_x u (0,t)=
   \d_x u (1,t) = 0.
  \end{equation}

By a change of variables
 $t\to \tau$ via $t(\tau) = \sqrt{R(x)}\tau$
 we transform the problem (\ref{PDE}-\ref{BCpde}) into the standard form given in (\ref{P}).
 The function $c$ in (\ref{P}) equals $A$ and therefore is of Heaviside
 type and the function $b$ is then given by
 $b(x,t)= P(R(x)t)$ and its regularity properties
 depend on the assumptions on $P$ and $R$.

We show in Section 2 that standard functional analytic techniques reach as far as the following: boundedness of $c$ and $b$
together with sufficient (spatial Sobolev) regularity of the initial values $f_1, f_2$  as well as of $g$ ensure existence of a unique solution $u\in L^2(0,T; H^2((0,1)))$ to (\ref{P}).
However, the prominent case $ P = P_0 + P_1 \delta (t-t_0)$ is clearly not covered by such a result.
Moreover, since both $c$ and $\d^2_xu$ are distributions their product is not always defined. In order to allow for these stronger singularities one needs to go beyond distributional solutions.


 \subsection{Nonlinear theory of generalized
 functions}\label{introCol}
  We will set up and solve problem \eqref{P}
  in the framework of algebras of generalized functions on the domain $X_T = (0,1) \times (0,T)$ (with $T > 0$).  To be more specific we employ a variant of the Colombeau spaces based on $L^2$-norm estimates as introduced in \cite{BO:92} (and applied in \cite{GH:04}). Needless to say that the general theory is set-up on arbitrary open subsets of $\R^d$ or even on smooth manifolds, but we use this opportunity to introduce notation appropriate for our specific model (cf.\ \cite{Colombeau:83,Colombeau:84,GKOS:01,O:92}).
  The
  basic objects defining our generalized functions are regularizing
  families $(u_{\eps})_{\eps\in (0,1]}$ of smooth functions $u_{\eps}\in
  H^{\infty}(X_T)$, the space of smooth functions on $X_T$ all of whose derivatives belong to $L^2$. To simplify notation we will henceforth write $(u_{\eps})_{\eps}$
  to mean $(u_{\eps})_{\eps\in (0,1]}$.

We single out the following subalgebras:

  {\it Moderate families}, denoted by ${\cal E}_{M,H^{\infty}(X_T)} $, are defined
  by the property
  $$
  \forall \al\in\N_0^n, \exists p\geq 0:
  \|\d^{\al}u_{\eps}\|_{L^2(X_T)}= O(\eps^{-p}), \quad \text{ as }
  \eps\to 0.
  $$

 {\it Null families}, denoted by ${\cal N}_{H^{\infty}(X_T)} $, are the families in ${\cal E}_{M,H^{\infty}(X_T)} $
  satisfying
 $$
  \forall q\geq 0: \|u_{\eps}\|_{L^2(X_T)} = O(\eps^q) \quad \text{ as }
  \eps\to 0.
 $$
 Hence moderate families satisfy $L^2$ estimates with at most
 polynomial divergence as $\eps\to 0$, together with all
 derivatives,  while null families vanish very rapidly as $\eps \to 0$.
 For the latter, one can show that, equivalently, all derivatives
 satisfy estimates of the same kind (cf. \cite[Proposition 3.4(ii)]{Garetto:05}). The
 null families form a differential ideal in the collection of moderate families.

 The {\it Colombeau algebra} is the factor algebra
 $$
  {\cal G}_{H^{\infty}(X_T)} = {\cal E}_{M, H^{\infty}(X_T)}/{\cal N}_{H^{\infty}(X_T)}.
 $$

We will occasionally use the notation $[(u_\eps)_\eps]$ for the equivalence class  in
  ${\cal G}_{H^{\infty}(X_T)}$ with representative $(u_\eps)_\eps$.
  The algebra ${\cal G}_{H^{\infty}((0,1))}$ on the interval $(0,1)$ is defined in the same way and its
  elements can be considered being members of ${\cal G}_{H^{\infty}(X_T)}$ as well.

\begin{remark} \label{rem_restriction}
\begin{enumerate}
\item
  For any  $(u_{\eps})_{\eps}$ in ${\cal E}_{M,H^{\infty}(X_T)}$ we have smoothness up to the boundary, i.e.\
  $u_{\eps}\in C^{\infty}([0,1] \times [0,T])$ and therefore
  the restriction $u |_{t=0}$  of a generalized function  $u \in {\cal G}_{H^{\infty}(X_T)}$ to
 $t=0$ is well-defined by $u_{\eps}(\cdot,0)\in {\cal E}_{M,H^{\infty}((0,1))}$.
(As emphasized in \cite[Remark 2.2]{BO:92} this follows from general Sobolev space properties, since $X_T$ clearly is a Lipschitz domain; cf.\ \cite{AF:03}.)

\item If in addition to $v \in {\cal G}_{H^{\infty}((0,1))}$ we have for some representative $(v_\eps)_\eps$ of $v$ that  $v_\eps \in H_0^{2}((0,1))$, then $v_\eps(0) = v_\eps(1) = 0$ and $\d_x v_\eps (0) = \d_x v_\eps(1) = 0$ (cf.\ Remark \ref{rem_boundary} below). In particular,
$$
   v(0) = v(1) = 0 \quad\text{and}\quad \d_x v(0) = \d_x v(1) = 0
$$
holds in the sense of generalized numbers.
%

\item Note  that $L^2$-estimates for parametrized families $u_\eps \in H^\infty(X_T)$ always yield $L^\infty$-estimates of the same qualitative behavior with respect to $\eps$ (since $H^\infty(X_T) \subset C^\infty(\ovl{X_T}) \subset W^{\infty,\infty}(X_T)$).

\end{enumerate}
\end{remark}


As described in \cite{BO:92} the space $H^{-\infty}(\R^d)$, i.e.\ distributions of finite order, are embedded (as a linear space) into ${\cal G}_{H^{\infty}(\R^d)}$ by convolution regularization. Furthermore, by this embedding $H^\infty(\R^d)$ becomes a subalgebra of ${\cal G}_{H^{\infty}(\R^d)}$.

  Some generalized functions display so-called macroscopic or distribution aspects:
   We say that $u =[(u_{\eps})_{\eps}]\in {\cal G}_{H^{\infty}}$ is {\it associated with the
   distribution} $w\in\D'$, denoted by $u \approx w$, if for some (hence any) representative $(u_{\eps})_{\eps}$ of
   $u$ we have $u_{\eps}\to w$ in $\D'$, as $\eps\to 0$.



  Intrinsic regularity theory for Colombeau generalized functions has been
  started by defining (in
  \cite{O:92}) the subalgebra $\G^{\infty}$, which satisfies the crucial compatibility property
  $\G^{\infty}\cap \D'= C^{\infty}$. In this sense, it plays the same role for $\G$ as $C^{\infty}$ does for $\D'$. Similarly \cite{Hoermann:03} introduces the subalgebra ${\cal G}^{\infty}_{H^{\infty}}$ of \emph{regular} elements in ${\cal G}_{H^{\infty}}$ by the following condition:
  $u\in{\cal G}_{H^{\infty}}$ belongs to ${\cal G}_{H^{\infty}}^{\infty}$
  if and only if
\beq \label{Col_reg}
   \exists p\geq 0, \forall\al\in\N_0^n:\,
   \|\d^{\al}u_{\eps}\|_{L^2}=O(\eps^{-p}), \text{ as }\eps\to 0.
\eeq
  Note that $p$ can be chosen uniformly with respect to $\al$.


 \subsection{Precise constants in energy estimates from variational methods}

 Let $V,H$ be two complex, separable Hilbert spaces, where $V$ is densely embedded into $H$. Denote the norms in $V$ by $|\cdot|$ and in $H$ by $\|\cdot\|$. Thus, if $V'$ denotes the antidual of $V$, then $V \subset H \subset V'$ forms a Gelfand triple.

Let $a(t,.,.)$, $a_0(t,.,.,)$, and $a_1(t,.,.)$ ($t\in [0,T]$) be (parametrized) families of continuous sesquilinear forms on $V$ with
 \begin{equation}\label{a0a1}
 a(t,u,v)=a_0(t,u,v)+ a_1(t,u,v) \qquad \forall u, v \in V,
 \end{equation}such that the ``principal part'' $a_0$ and the remainder $a_1$ satisfy the following conditions:
\begin{enumerate}
\item for all $u,v \in V$: $t \mapsto a_0(t,u,v)$ is continuously differentiable $[0,T] \to \C$,

\item $a_0$ is Hermitian, i.e.\
     $a_0(t,u,v)= \overline{a_0(t,v,u)}$ for all $u,v\in V$,

\item there exist real constants $\la$ and $\al >0$ such that
\beq \label{coercivity}
    a_0(t,u,u) \geq \al |u|^2 - \la \|u\|^2
       \qquad \forall u\in V,\, \forall t\in [0,T],
\eeq

\item for all $u,v\in V$:  $t \mapsto a_1(t,u,v)$ is continuous in $[0,T] \to
\C$,

\item there is $C_1 \geq 0$ such that for all $t\in [0,T]$ and $u,v\in V$:  $|a_1(t,u,v)|\leq C_1 |u|\, \|v\|$.

\end{enumerate}

Note that writing $a'_0(t,u,v)=\frac{d}{dt} \big(a_0(t,u,v)\big)$ condition (i) implies the following: there exist nonnegative constants $C$ and $C_0$ such that for all $t\in [0,T]$ and $u,v\in V$
\beq \label{i_cons}
  |a_0(t,u,v)| \leq C |u|\,|v| \quad \text{and} \quad
    |a'_0(t,u,v)| \leq C_0 |u| \,|v|.
\eeq

As shown in \cite[Chapter XVIII]{DL:V5} the above conditions  guarantee unique solvability of the abstract variational problem as described by the following Theorem.

\begin{theorem}\label{ThmP1} Let $a(t,.,.)$ ($t \in [0,T]$) satisfy conditions (i)-(v).
Let  $u_0\in V$, $u_1\in H$, and $f\in L^2((0,T),H)$.
Then there exists a unique $u \in L^2((0,T),V)$ satisfying the regularity conditions
\beq \label{var_reg}
 u'= \frac{du}{dt}\in L^2((0,T),V) \quad \text{ and } \quad
   u''=\frac{d^2u}{dt^2}\in L^2((0,T),V')
\eeq
and solving the abstract initial value problem
\begin{align}
    &\dis{u''(t)}{v} + a(t,u(t),v)=\dis{f(t)}{v}
     \qquad \forall v\in V, \, \forall t \in (0,T) \label{var_equ}\\
     & u(0)=u_0,\quad u'(0)=u_1. \label{var_ini}
   \end{align}
(Note that (\ref{var_reg}) implies that $u \in C([0,T],V)$ and $u' \in C([0,T],V')$. Hence it makes sense to evaluate $u(0)\in V$ and $u'(0) \in V'$ and (\ref{var_ini}) claims that these equal $u_0$ and $u_1$, respectively.)
\end{theorem}

  The prove of this theorem in \cite[Chapter XVIII]{DL:V5} (see also \cite[Chapter III, Section 8]{LM:72}) proceeds along the following lines:
  first, one shows that $u$
  satisfies  a priori (energy) estimates  which immediately imply
  uniqueness of the solution
  and then deduces existence of a solution by a Galerkin approximation.
  To identify the precise dependence of all constants in the a priori estimates we adapt the corresponding part of the proof to our case (which happens to be simpler than in \cite{DL:V5}, whereas \cite{LM:72} does not cover our situation).

  \begin{proposition}\label{propEE}
   Let $u$ be a solution to the abstract variational problem (\ref{var_reg}-\ref{var_ini}), then for all $t \in [0,T]$
   \begin{equation}\label{EE}
   |u(t)|^2 + \|u'(t)\|^2 \leq
    ( D_T \,|u_0|^2 + \|u_1\|^2 + \int_0^t \|f(\tau)\|^2 \,d\tau) \cdot
     \exp(t\, F_T)
  \end{equation}
  where the constants $D_T$ and $F_T$ are given in terms of the constants $C,C_0,C_1, \al, \lambda$ occurring in conditions (iii),(v), and \eqref{i_cons} explicitly by
$$
   D_T = (C + \la (1+T)) /\min(\alpha,1) \quad \text{ and }\quad
   F_T = \max \{C_0 + C_1, C_1+T+2 \} / \min(\alpha,1).
$$

  \end{proposition}
  \begin{proof}
  We put $v=u'(t)$ in (\ref{var_equ}) and obtain
  \begin{equation}\label{varEq1}
   \dis{u''(t)}{u'(t)} +  a(t,u(t),u'(t)) = \dis{f(t)}{u'(t)}.
  \end{equation}
  By sesquilinearity and (ii) we have
  $$
   \diff{t} \|u'(t)\|^2 = \diff{t} \big(\dis{u'(t)}{u'(t)}\big) =
  \dis{u''(t)}{u'(t)} + \dis{u'(t)}{u''(t)} = 2 \Re(\dis{u''(t)}{u'(t)})
  $$
  and $\overline{a(t,u,v)} = a_0(t,v,u) + \overline{a_1(t,u,v)}$.
Thus taking real parts in (\ref{varEq1}) we obtain
  \begin{align}\label{varEq2}
   \diff{t} \|u'(t)\|^2 & + a_0(t,u'(t),u(t)) +
   a_0(t,u(t),u'(t))\nonumber \\
    + &\overline{a_1(t,u(t),u'(t))} + a_1(t,u(t),u'(t))
   = 2 \Re \dis{f(t)}{u'(t)}.
  \end{align}

For any $v \in H^1((0,T),V)$
 \begin{equation*}
  \diff{t} \big( a_0(t,u(t),v(t)) \big) =
  a_0'(t,u(t),v(t))+ a_0(t,u'(t),v(t)) +
  a_0(t,u(t),v'(t)),
  \end{equation*}
 hence (\ref{varEq2}) implies
  \begin{align}\label{varEq4}
   \diff{t} \Big( \|u'(t)\|^2 +  a_0(t,u(t),u(t))\Big)  =
    - 2 \Re\, a_1(t,u(t),u'(t)) + a_0'(t,u(t),u(t)) + 2 \Re \,\dis{f(t)}{u'(t)}
  \end{align}
Integration (\ref{varEq4}) we obtain
  \begin{align}\label{varEq5}
  LHS(t) := & \|u'(t)\|^2  +  a_0(t,u(t),u(t)) \nonumber \\
  = & \|u'(0)\|^2 + a_0(0,u(0),u(0))
   - 2 \int_0^t \Re\, a_1(\tau,u(\tau),u'(\tau))d\tau \nonumber \\
   & + \int_0^t a_0'(\tau,u(\tau),u(\tau))d\tau + 2 \int_0^t\Re \,\dis{f(\tau)}{u'(\tau)}d\tau =: RHS(t).
  \end{align}
  Further, using the conditions (i) and (v), initial conditions  (\ref{var_ini}) and the inequality $2ab\leq a^2 + b^2$
  we obtain that the right-hand side of (\ref{varEq5}) can
  be estimated in the following way:
  \begin{align}
  |RHS(t)|& \leq  \|u_1\|^2 + C |u_0|^2 + 2C_1\int_0^t |u(\tau)| \, \|u'(\tau)\| \,d\tau +
       C_0\int_0^t |u(\tau)|^2 \, d\tau \nonumber\\
  &\quad + \int_0^t \|f(\tau)\|^2 \,d\tau +
  \int_0^t \|u'(\tau)\|^2d\tau \nonumber \\
  &\leq C\, |u_0|^2 + \|u_1\|^2 +
  (C_0 + C_1)\int_0^t |u(\tau)|^2 \, d\tau
  + \int_0^t \|f(\tau)\|^2\, d\tau +
  (1+ C_1)\int_0^t \|u'(\tau)\|^2\,d\tau.
    \label{upper_bound}
  \end{align}
  Condition (\ref{coercivity}) implies that the left-hand side  of (\ref{varEq5}) has the lower bound
  $$
   LHS(t) \geq \al |u(t)| - \la\, \|u(t)\|^2 + \|u'(t)\|^2,
  $$
  and therefore
  \begin{equation}\label{varEq6}
   \al \|u(t)\|^2 + |u'(t)|^2 \leq |RHS(t)| + \la |u(t)|^2.
  \end{equation}

We claim that for all $t \in [0,T]$:
  \begin{equation}\label{uodt}
   |u(t)|^2 \leq (1+t)(\|u_0\|^2 + \int_0^t |u'(\tau)|d\tau ).
  \end{equation}
  Indeed, since $ u(t) = u_0 + \int_0^t u'(\tau)d\tau$ and
  $\|u_0\| \leq |u_0|$  we have
  \begin{align*}
   \|u(t)\|^2 &= \dis{u(t)}{u(t)}= \|u_0\|^2 + 2\Re \int_0^t
   \dis{u_0}{u'(\tau)d\tau} +
   \dis{\int_0^tu'(\tau)d\tau}{\int_0^tu'(s)ds}\\
   &\leq \|u_0\|^2 + \int_0^t \|u_0\|^2d\tau +
   \int_0^t \|u'(\tau)\|^2 d\tau + \int_0^t\int_0^t
   |\dis{u'(\tau)}{u'(s)}| ds\, d\tau\\
   &\leq \|u_0\|^2 + \int_0^t \|u_0\|^2d\tau +
   \int_0^t \|u'(\tau)\|^2 d\tau +
   \int_0^t\int_0^t
   \frac{\|u'(\tau)\|^2 + \|u'(s)\|^2}{2} ds\, d\tau\\
   &\leq  \|u_0\|^2 + t\, \|u_0\|^2 +
   \int_0^t \|u'(\tau)\|^2 d\tau +
   t\int_0^t \|u'(\tau)\|d\tau\\
   &\leq (1+t)(\|u_0\|^2 + \int_0^t \|u'(\tau)\|^2 d\tau)
   \leq (1+t)(|u_0|^2 + \int_0^t \|u'(\tau)\|^2 d\tau).
  \end{align*}

 Let $\be:= \min \{\al,1\}$. Combining (\ref{upper_bound}), (\ref{uodt}), and (\ref{varEq6}) we arrive at
 \begin{align*}
   \be (|u(t)|^2 + \|u'(t)\|^2)&\leq
   ( C + \la (1+t)) |u_0|^2 +
   \|u_1\|^2 + \int_0^t \|f(\tau)\|^2 \,d\tau +\\
   &(C_0+C_1)\int_0^t
   |u(\tau)|^2 \,d \tau + (2+t+C_1)\int_0^t
   \|u'(\tau)\| \,d\tau\\
       &\leq (C + \la (1+T)) |u_0|^2 + \|u_1\|^2 +
       \int_0^T \|f(\tau)\|^2 d\tau \\
       &+
       \max \{C_0 + C_1, C_1+T+2 \}
       \int_0^t (|u(\tau)|^2 + \|u'(\tau)\|^2)d\tau.
 \end{align*}
Dividing by $\be>0$ yields
 $$
 \Phi(t)\leq D_T \, |u_0|^2 + \|u_1\|^2 +
 \int_0^T \|f(\tau)\|^2 d\tau +
 F_T \int_0^t \Phi(\tau)d\tau,
 $$
 where $\Phi(t)= |u(t)|^2 + \|u'(t)\|^2$ (and $D_T = (C + \la (1+T)) /\beta$ and $F_T = \max \{C_0 + C_1, C_1+T+2 \} / \beta$ as in the statement of the Proposition).
 Gronwall's lemma now implies that
 $$
  \Phi(t) \leq \big( D_T\, |u_0|^2 + \|u_1\|^2 + \int_0^T
 \|f(\tau)\|^2d\tau \big) \cdot \exp(t\, F_T).
 $$
 \ \hfill \end{proof}


 \section{Weak solutions for $\mathbf{L^\infty}$-coefficients}

  Let $H := L^2((0,1))$ with the standard scalar product
   $\dis{u}{v}=\int_0^1u(x)\overline{v(x)}dx$  and $L^2$-norm
   denoted by $\|\cdot\|$. Let $V$ be the Sobolev space
   $H^2_0((0,1))$, which is the completion of the space
   $C^\infty_c((0,1))$ (compactly supported smooth functions) with respect to the norm
$\|u\|_{2}
  = ( \sum_{k= 0}^2 \|u^{(k)}\|^2)^{1/2}$ (and inner product $(u,v) \mapsto \sum_{k= 0}^2 \dis{u^{(k)}}{v^{(k)}}$).
  Then  $V'= H^{-2}((0,1))$, which consists of distributional derivatives  up to second order of functions in $L^2((0,1))$,
  and  $V\hookrightarrow H \hookrightarrow V' $
   forms a Gelfand triple.

Let $c \in L^\infty((0,1))$ be real-valued and $b \in C([0,T],L^\infty((0,1)))$.
 For $t \in [0,T]$ we define the sesquilinear forms $a(t,\cdot,\cdot)$, $a_0(t,\cdot,\cdot)$, and $a_1(t,\cdot,\cdot)$ on $V\times V$ by
  \begin{equation}\label{sesforma}
   a_0(t,u,v) = \dis{c \, \d_x^2u}{\d_x^2v}, \quad
   a_1(t,u,v) = \dis{b(t)\d_x^2u}{v}
   \quad (u,v\in V)
  \end{equation}
and
\beq \label{sesform2}
   a(t,u,v) = a_0(t,u,v) + a_1(t,u,v).
\eeq
Clearly, $a_0$ is Hermitian (since $c$ is real).

Note that the $L^\infty$-properties of $c$ and $b$ are necessary in order to have the above sesquilinear forms defined and continuous on all of $V$. In the application to our model (\ref{P}) the functions $c$ and $b$ play the role of the coefficients (with the abuse of notation in identifying $b(t)(x)$ with $b(x,t)$) and hence the required $L^\infty$-condition is satisfied for $c$ but typically not for $b$,
which we want to allow to be a measure or more general distribution. This poses the main difficulty in the analysis of the model. Our strategy will be to regularize the problem by smoothing of the coefficients and data and then analyze the corresponding family of solutions and show that they constitute a generalized solution.

As a preparatory result, we need precise statements and estimates for the solutions to the regularized problems. In fact, for the basic result we will have to assume considerably less regularity than actually needed in the theory applied later on. The advantage is that it illustrates how far we could go by classical functional analytic methods concerning lowest possible regularity of the coefficients.
At this stage we make the hypotheses
  \begin{equation}\label{HypothesisOn_c_and_b}
   c\in L^{\infty}(0,1),\quad b\in C([0,T],L^{\infty})
  \end{equation}
and (motivated by the specification of the bending stiffness in Subsection 1.1) that there exist $c_1 > c_0 > 0$ such that
 \begin{equation}\label{addHypothesisOn_c}
   0 < c_0\leq c(x)\leq c_1 \qquad \text{(for almost every $x$)}.
  \end{equation}


\begin{remark}{{\bf (About boundary conditions)}}\label{rem_boundary}
Note that any solution of the partial differential equation and the initial conditions in \eqref{P} with $u(.,t) \in H^2_0((0,1))$ for all $t \in [0,T]$ automatically implements
the boundary conditions of (\ref{P}) in the following sense:  $H^2_0((0,1))$ is continuously embedded in the subspace of $C^1([0,1])$ of functions with vanishing values and derivatives at the boundary (\cite[Corollary 6.2]{Wloka:87}); hence we must have $u(0,t) = u(1,t) = 0$ as well as $\d_x u(0,t) = \d_x u(1,t) = 0$.
\end{remark}

\begin{theorem}\label{ThmClassic}
Let $b,c$ be as in (\ref{HypothesisOn_c_and_b}) and (\ref{addHypothesisOn_c}) and the sesquilinear forms $a(t,.,.)$ (for $t \in [0,T]$) be defined by (\ref{sesforma}) and (\ref{sesform2}).
If $f_1\in H_0^2((0,1))$, $f_2\in L^2((0,1))$,
 and $g\in L^2((0,T), L^2(0,1))$, then
 there exists a unique $u\in L^2((0,T),H^2_0((0,1)))$
 satisfying
\beq \label{sol_reg}
    u' = \diff{t}u \in L^2((0,T),H^2_0((0,1))),\qquad
    u'' = \frac{d^2}{dt^2}u\in L^2((0,T),H^{-2}((0,1)))
\eeq
 and solving the initial value problem
 \begin{align}
   &\dis{u''(t)}{v} + a(t,u(t),v) = \dis{g(t)}{v}, \quad \quad \forall  v\in H^2_0((0,1)),\; t \in (0,T),\label{vf}\\
   & u(0)=f_1,\quad u'(0) = f_2. \label{vf_ini}
 \end{align}
(Regarding the precise meaning of the initial conditions (\ref{vf_ini}) we refer to the corresponding remark in Theorem \ref{ThmP1}.)
 \end{theorem}

  \begin{proof}
  We show that the sesquilinear form $a$ satisfies all hypotheses of Theorem \ref{ThmP1}.
  First, for all $u, v \in H^2_0((0,1))$ we clearly have that $a_0$ is Hermitian (since $c$ is real) and continuously differentiable  (by independence of $t$) and that $a_1$ is continuous with respect to $t$ and satisfy the estimates
\begin{align*}
  |a_0(t,u,v)| &\leq
     \|c\|_{L^{\infty}((0,1))}\|D^2u\|\|D^2v\|
     \leq C \, \|u\|_2\|v\|_2 \\
   |a_1(t,u,v)| &\leq
  \|b\|_{L^{\infty}((0,1)\times(0,T))}\|D^2u\|\|v\|
  \leq C_1 \, \|u\|_2\|v\|_2,
\end{align*}
  where $C = \|c\|_{L^{\infty}((0,1))}$ and $C_1 = \|b\|_{L^{\infty}((0,1)\times(0,T))}$. Hence conditions (i), (ii), (iv), and (v) are met.

To prove property (iii) we appeal to Ehrling's
  lemma (\cite[Theorem 7.4]{Wloka:87}
  or \cite[Chapter 2, Proposition 2.11]{CP:82})
  which implies that for each  $\del > 0$  there exists
  a real constant $C_{\del}$ such that for all $u\in H^2_0((0,1))$
\begin{equation}\label{poe}
  \|u\|^2_1 \leq \del\, \|u\|_{2}^2 + C_{\del}\, \|u\|^2.
  \end{equation}
Choosing $\del = 1/2$ and employing
  (\ref{addHypothesisOn_c})
  we obtain for all $u\in H^2_0((0,1))$
\begin{multline*}
 a_0(t,u,u)= \dis{c\, \d_x^2u}{\d_x^2u}\geq
     c_0 \| \d_x^2u\|^2 =
     c_0 \|u\|^2_2 - c_0 \, \|u\|_1^2  \\
     \geq c_0 \|u\|^2_2 -
     c_0(\frac{1}{2} \|u\|_2^2 + C_{1/2}\|u\|^2)
     \geq \frac{c_0}{2} \|u\|^2_2 -
     c_0 C_{1/2} \|u\|^2 \\ =
     \al \|u\|^2_2 - \la \|u\|^2,
\end{multline*}
  where $\al = c_0 /2$ and $\la =  C_{1/2} c_0$. Thus we have shown also (iii).
\end{proof}

\begin{remark}\label{rem_const_distr_sol}
\begin{enumerate}
\item
 For later reference we give the precise dependence of
 all constants appearing in the energy
 estimate (\ref{EE})
$$
  \|u(t)\|_2^2 + \|u'(t)\|^2 \leq
    ( D_T \,\|u_0\|_2^2 + \|u_1\|^2 + \int_0^t \|g(\tau)\|^2 \,d\tau) \cdot  \exp(t\, F_T)
$$
on the coefficients $b$ and $c$. We recall that
$$
   D_T = (C + \la (1+T)) /\min(\alpha,1) \quad \text{ and }\quad
   F_T = \max \{C_0 + C_1, C_1+T+2 \} / \min(\alpha,1),
$$
where we now have
\begin{align*}
  C & = \|c\|_{L^{\infty}((0,1))},\quad
  C_0  = 0,\quad
  C_1 = \|b\|_{L^{\infty}((0,1)\times(0,T))}\\
  \al &= \frac{c_0}{2},\quad
  \la =  C_{1/2}\, c_0.
\end{align*}
\item According to (\ref{sol_reg}) the solution $u$ belongs to $C^1([0,T],H^{-2}((0,1))) \hookrightarrow \D'((0,1)\times(0,T))$ and in case of smooth coefficients $b$ and $c$ is a
distributional solution to the partial differential equation
$$
   \d_t^2 u + \d^2_x(c \,\d^2_x)u + b \,\d_x^2 u = g
   \qquad \text{in } \D'((0,1)\times(0,T)).
$$
Moreover, if $u_0$, $u_1$, and $g$ are smooth, then $u$ is a classical smooth solution to the partial differential equation. (We prove this in course of the proof of Theorem \ref{main_thm} below; cf.\ Remark \ref{after_main_thm}(iii).)

\item In view of the statement in (ii) it might be interesting to note that the partial differential operator in our model does not fall into any of the standard types of linear differential operators of distribution theory. To make this precise it suffices to consider the case of constant coefficients $c > 0$ and $b \in \R$. Thus we have the operator
$$
  P = \d_t^2 + c\, \d_x^4 + b\, \d_x^2
$$
with symbol
$p(\xi,\tau) = - \tau^2 + c\, \xi^4 - b\, \xi^2$
and principal symbol
$p_2(\xi,\tau) = c\, \xi^4$.
Then the  following holds (for clarity, with slight logical redundance in the statements):
\begin{enumerate}
\item Clearly, $P$ is not elliptic (\cite[p.\ 176]{DL:V2}), since $\Char(P) = \{ (\xi,\tau) \in
\R^2 \setminus\{0\} : p_2(\xi,\tau) = 0\} = \{(0,\tau) : \tau\not=0\}$ is not empty.

\item $P$ is not parabolic (with respect to forward time, i.e.\ with respect to $H = {t \geq 0}$; \cite[pp. 202-203, Corollary 2]{DL:V2}), since with $\xi > 0$
sufficiently large such that $\tau(\xi) := c\, \xi^4 - b\,  \xi^2 > 0$ we have $|\xi| + |\tau(\xi)| \to \infty$ ($\xi \to
\infty$), while $p(\tau(\xi) \cdot (0,1) + i(\xi,0)) = 0$.
(Note that, in contrast, $\d_t + \d_x^4 - \d_x^2$ is parabolic.)

\item $P$ is not hyperbolic (with respect to forward time; \cite[pp. 190-193, Prop. 16]{DL:V2}), since $(0,1) \in \Char(P)$.

\item Moreover, $P$ is not hypoelliptic (thus cannot be parabolic or elliptic), since the equivalent condition \cite[(4.11), p.\ 233]{DL:V2} can be shown to be false (cf.\ \cite[Theorem 1 on p.233 below]{DL:V2}). With $k = (0,1)$ (i.e.\ the partial $\tau$-derivative) and $\xi > 0$
sufficiently large to ensure $\tau(\xi) := c\, \xi^4 - b\,  \xi^2 > 1$ we obtain $|\d_{\tau} p (\xi,\tau(\xi)) | = 2 \tau(\xi) \to
\infty$ (as $\xi \to \infty$), whereas $p(\xi,\tau(\xi)) = 0$ for large $\xi > 0$.

\item Of course, $P$ is also not of Schr\"odinger type (in the sense of \cite[pp.\ 620]{DL:V5}). However, if $b = 0$ one can write $P = (\d_t - i \sqrt{c}\, \d_x^2)\cdot(\d_t + i \sqrt{c}\, \d_x^2)$, which is a product of two
Schr\"odinger operators.

\end{enumerate}

\end{enumerate}
\end{remark}
\section{Colombeau generalized solutions}

We now establish existence and uniqueness of a
 generalized solution to problem (\ref{P}), where the coefficients $b,c$ as well as the data
  $f_1,f_2$ and $g$ are Colombeau generalized functions.
This means that we find a unique solution $u \in \G_{H^\infty(X_T)}$ to the partial differential equation
$$
  \d^2_t u + Q(t,x,\d_x) u =  g
    \qquad \text{on } X_T = (0,1)\times(0,T),
$$
 where $g \in \G_{H^\infty(X_T)}$  and $Q$ denotes the differential operator on $\G_{H^\infty(X_T)}$ which acts on representatives by
$$
 (u_{\eps})_{\eps} \mapsto
   (\d^2_x(c(x)\d^2_x(u_{\eps}))_\eps  +
     b(x,t)\d_x^2(u_{\eps}))_{\eps} =: Q_\eps u_\eps,
$$
with initial conditions
$$
  u|_{t=0} = f_1 \in \G_{H^\infty((0,1))}, \quad
  \d_tu|_{t=0} = f_2 \in \G_{H^\infty((0,1))}.
$$
Recall that the initial conditions are to be understood in the sense of Remark \ref{rem_restriction}(i).

Furthermore, the boundary conditions read
 $$
  u|_{x=0}=u|_{x=1}=0,\quad \d_x u|_{x=0} =
  \d_x u|_{x=1}=0.
$$
Thanks to Remark \ref{rem_restriction}(ii) these are automatically satisfied if we can show that for all $t \in [0,T]$ we have $u_\eps(t) \in H^2_0((0,1))$ for some representative $(u_\eps)_\eps$ of $u$.

As in the classical specification with we have to impose a compatibility condition concering initial and boundary values
\beq \label{compatibility}
   f_1(0) = f_1(1) = 0
\eeq
(which is an equation in generalized numbers). We observe that for any $f_1 \in \G_{H^\infty((0,1))}$ condition \eqref{compatibility} implies that there exists a representative $(f_{1,\eps})_\eps$ of $f_1$ such that $f_{1,\eps} \in H^2_0((0,1))$ for all $\eps \in \,]0,1[$. (If $(\bar{f}_\eps)_\eps$ is an arbitrary representative, put $n_\eps := \bar{f}_{1,\eps}(0)$, $m_\eps := \bar{f}_{1,\eps}(1)$ and consider the new representative $f_{1,\eps}(x) := \bar{f}_{1,\eps}(x) - n_\eps - (m_\eps - n_\eps)\, x$ instead.)

\subsection{Existence and uniqueness}

As in Section 2 we impose a condition on $c$, which is motivated by the intended properties of the bending stiffness: There exist real constants $c_1 > c_0 > 0$ such that $c\in {\cal G}_{H^{\infty}(0,1))}$
possesses a representative $(c_{\eps})_\eps$ satisfying
\beq \label{Col_stiff}
  0 < c_0 \leq c_{\eps}(x) \leq c_1
   \qquad \forall x\in (0,1), \forall \eps \in\, ]0,1].
\eeq
(Note that then any other representative
$(\widetilde{c_{\eps}})_\eps$ of $c$ satisfies $\frac{c_0}{2}\leq \widetilde{c_{\eps}}(x)\leq c_1+1$ for all $x \in (0,1)$ and  $0 < \eps < \eps_0$ with some $\eps_0 \in\, ]0,1]$.)
A weaker condition on $c$ is briefly discussed in Remark \ref{after_main_thm}(i) below.

As a technical assumption we will also require that $b$ is of $L^\infty$-log-type (similar to \cite{O:89}), which means that we have for some (hence any) representative $(b_\eps)_\eps$ of $b$ there exist $N\in\N$ and $\eps_0 \in ]0,1]$ such that
\beq \label{log_type}
  \| b_\eps \|_{L^\infty(X_T)}
    \leq N\cdot \log(\frac{1}{\eps})
      \qquad (0 < \eps \leq \eps_0).
\eeq
As observed in \cite[Proposition 1.5]{O:89} log-type regularizations of distributions are easily obtained via convolution with lograithmically scaled mollifiers.

 \begin{theorem}\label{main_thm}
  Let $b\in {\cal G}_{H^{\infty}(X_T)}$ be of $L^\infty$-log-type and $c\in {\cal G}_{H^{\infty}(0,1))}$ satisfy \eqref{Col_stiff}.
 For any  $f_{1} \in {\cal G}_{H^{\infty}((0,1))}$ satisfying \eqref{compatibility}, $f_{2}\in {\cal G}_{H^{\infty}((0,1))}$, and
 $g\in {\cal G}_{H^{\infty}(X_T)}$
there is a unique solution $u \in {\cal G}_{H^{\infty}(X_T)}$
to the initial-boundary value problem (\ref{P}).
 \end{theorem}

\begin{proof} {\bf Existence:}\enspace \enspace
Let $(b_\eps)_\eps$ represent $b$ and $(c_{\eps})_\eps$ be a representative of $c$ satisfying \eqref{Col_stiff}. Denote by  $(f_{1\eps})_\eps$, $(f_{2\eps})_\eps$, and $(g_{\eps})_\eps$ representatives of $f_1$,$f_2$, and $g$, respectively. In addition, we may assume that $f_{1,\eps} \in H^2_0((0,1))$ for all $\eps \in \,]0,1[$ (see the discussion following \eqref{compatibility}).

Let $\eps \in\, ]0,1]$ be arbitrary. By  Theorem \ref{ThmClassic} and Remark \ref{rem_const_distr_sol}(ii) we obtain a unique function
$u_{\eps}\in H^1((0,T),H^2_0((0,1)))
\cap H^2((0,T),H^{-2}((0,1)))$ which solves
 \begin{align}\label{Peps}
  P_{\eps}u_{\eps} &: = \d^2_tu_{\eps} +
    Q_{\eps}(t,x,\d_x)u_{\eps} = g_{\eps}
      \qquad \text{on } X_T,\\
 &u_{\eps}|_{t=0} = f_{1\eps},
   \quad \d_tu_{\eps}|_{t=0} = f_{2\eps}.\nonumber
\end{align}
In particular, we have $u_{\eps}\in C^1([0,T],H^{-2}((0,1))) \cap
  C([0,T],H^2_0((0,1)))$.

Moreover, from Proposition \ref{propEE} and Remark \ref{rem_const_distr_sol}(i) we deduce the energy estimate
  \begin{equation}\label{EEeps1}
   \|u_{\eps}(t)\|_2^2 + \|u_{\eps}'(t)\|^2 \leq
    \big( D_T^\eps\, \|f_{1\eps}\|_2^2 + \|f_{2,\eps}\|^2
    + \int_0^t \|g_{\eps}(\tau)\|^2\,d\tau \big)
    \cdot \exp(t\, F_T^\eps),
  \end{equation}
where with some $N$ we have for small $\eps >0$
\begin{align}
  D_T^\eps &= (\|c_\eps\|_{L^{\infty}} +
    \la (1+T)) /\min(\alpha,1)
    = O(\|c_\eps\|_{L^{\infty}}) = O(1)
      \label{const_D_eps} \\
  F_T^\eps &= \max \{\|b_\eps\|_{L^{\infty}},
  \|b_\eps\|_{L^{\infty}} +T+2 \} / \min(\alpha,1)
  = O(\|b_\eps\|_{L^{\infty}})
  = O(\log(\eps^{-N})), \label{const_F_eps}
\end{align}
since $\al = c_0/2$ and $\la = C_{1/2}\, c_0$ are independent of $\eps$.
Therefore moderateness of the initial conditions $f_{1\eps}$, $f_{2\eps}$ and of $g_{\eps}$ in (\ref{EEeps1})
  yields that there exists $M$ such that for
  small $\eps > 0$
  \begin{equation}\label{EEeps}
  \|u_{\eps}\|^2_{L^2(X_T)} +
  \|\d_x u_{\eps}\|^2_{L^2(X_T)}
  + \| \d_x^2 u_{\eps}\|^2_{L^2(X_T)}
  + \|\d_t u_{\eps}\|^2_{L^2(X_T)} =
  O (\eps^{-M})\quad (\eps\to 0).
  \end{equation}

We will proceed to show that the family $(u_{\eps})_{\eps}$ belongs to ${\cal E}_{H^{\infty}(X_T)}$. Then by construction its class $u$ in ${\cal G}_{H^{\infty}(X_T)}$ defines a solution to the initial value problem. It remains to prove the following properties:
 \begin{itemize}
  \item[1.)]  For all $\eps \in\, ]0,1]$ the function
  $u_{\eps }$ is smooth, i.e.\ $u_{\eps}\in C^{\infty}(X_T)$.
  \item[2.)] Moderateness, i.e.\ for all $l,k\in\N$ there is some $M\in\N$ such that for small $\eps > 0$
\beq \tag{$T_{l,k}$} \label{T_lk} \qquad
  \|\d^l_t\d^k_xu_{\eps}\|_{L^2(X_T)} = O(\eps^{-M}).
\eeq
 Note that \eqref{EEeps} already yields \eqref{T_lk} for $(l,k) \in \{ (0,0),(1,0),(0,1),(0,2)\}$.
 \end{itemize}

 \noindent\emph{Step 1:}  Differentiating (\ref{Peps}) (considered as an equation in $\D'((0,1)\times(0,T))$) with respect to $t$ we obtain
 $$
 P_{\eps}(\d_tu_{\eps}) = \d_tg_{\eps} - \d_t b(x,t)\d_x^2u_{\eps} =: \tilde{g}_\eps \in H^1((0,T),L^2((0,1))),
 $$
 since
 $\d_tg_{\eps}\in H^{\infty}(X_T)$,
 $\d_t b(x,t)\in H^{\infty}(X_T) \subset W^{\infty,\infty}(X_T)$
 and  $\d^2_xu_{\eps}\in H^1((0,T),L^2((0,1)))$.
Furthermore, since $Q_\eps$ depends smoothly on $t$ as a differential operator in $x$ and $u_\eps(0) = f_{1,\eps} \in H^{\infty}((0,1))$ we have
\begin{align*}
    (\d_t u_{\eps})(.,0) & = f_{2,\eps} =:
      \tilde{f}_{1,\eps} \in H^{\infty}((0,1)),\\
   (\d_t(\d_tu_{\eps}))(.,0) & =
   g_{\eps}(.,0) - Q_{\eps}(u_{\eps}(.,0)) =
     g_{\eps}(.,0) - Q_{\eps} f_{1,\eps} :=
       \tilde{f}_{2,\eps} \in H^{\infty}((0,1)).
\end{align*}
Hence $\d_tu_{\eps}$ satisfies an initial value problem for the partial differential operator $P_\eps$ as in \eqref{Peps} with initial data
  $\tilde{f}_{1,\eps}$, $\tilde{f}_{2,\eps}$ and right-hand side $\tilde{g}_{\eps}$ instead. However, this time we have to use $V = H^2((0,1))$ (replacing $H^2_0((0,1))$) and $H = L^2((0,1))$ in the abstract setting, which still can serve to define a Gelfand triple $V\hookrightarrow H \hookrightarrow V'$ (cf.\ \cite[Theorem 17.4(b)]{Wloka:87}) and thus allows for application of Theorem \ref{ThmP1} and the energy estimate \eqref{EE} (with precisely the same constants).

 Therefore we obtain
 $\d_tu_{\eps}\in H^1([0,T],H^2((0,1)))$,
  i.e.\ $u_{\eps}\in H^2((0,T),H^2((0,1)))$ and from the
  variants of \eqref{EEeps1} (with exactly the same constants $D_T^\eps$ and $F_T^\eps$) and \eqref{EEeps} with $\d_t u_\eps$ in place of $u_\eps$ that for some $M$ we have
 \begin{equation}
   \| \d_t u_{\eps}\|^2_{L^2(X_T)} +
  \|\d_x \d_t u_{\eps}\|^2_{L^2(X_T)}
  + \| \d_x^2 \d_t u_{\eps}\|^2_{L^2(X_T)}
  + \|\d_t^2 u_{\eps}\|^2_{L^2(X_T)} =
  O (\eps^{-M})\quad (\eps\to 0).
  \end{equation}
 Thus we have proved \eqref{T_lk} with $(l,k) = (2,0), (1,1), (1,2)$ in addition to those obtained from \eqref{EEeps} directly.

Differentiating (\ref{Peps}) a second time with respect to $t$ we obtain again an initial value problem for the same operator $P_\eps$ with solution $\d_t^2 u_\eps$: this time, the right-hand side reads
$$
   \tilde{\tilde{g}}_\eps :=
  \d_t^2 g_{\eps} - \d_t^2 b(x,t) \d_x^2u_{\eps}
  - \d_t b(x,t) \d_x^2\d_t u_{\eps}
$$
and the initial values are $\tilde{\tilde{f}}_{1,\eps}:=
   (\d_t^2 u_{\eps})(.,0) = \tilde{f}_{2,\eps}$ and
$$
    \tilde{\tilde{f}}_{2,\eps} :=
   (\d_t(\d_t^2 u_{\eps}))(.,0)  =
   \d_t g_{\eps}(.,0) - Q_{\eps}(\tilde{f}_{1,\eps})
   - \d_t b_\eps(.,0)\, \d_x^2 {f}_{1,\eps}.
$$
Note that the corresponding energy estimate involves again precisely the same constants $D_T^\eps$ and $F_T^\eps$ as in  \eqref{EEeps1} and we only have to insert the norms of the new initial values and right-hand side accordingly. Thus we obtain $u_{\eps}\in H^3((0,T),H^2((0,1)))$ and $L^2$-norm estimates proving \eqref{T_lk} also for the new pairs $(l,k) = (3,0),(2,1),(2,2)$ in addition.

Upon deriving a succession of similar initial value problems for $\d_t^3 u_\eps$, $\d_t^4 u_\eps$ etc.\ we arrive at
  \begin{equation}\label{smoothnes_and_reg_po_t}
   u_{\eps}\in H^{\infty}((0,T),H^2((0,1)))
  \end{equation}
and
\beq \label{T_all_times}
  \text{\eqref{T_lk} holds for all } l\in\N \text{ and } k=0,1,2.
\eeq

Using this information in the original equation (\ref{Peps}) gives
 $$
  Q_{\eps}u_{\eps} = g_{\eps} - \d_t^2 u_{\eps}\in
  H^{\infty}((0,T),H^2((0,1)).
 $$
 Since $Q_{\eps}$ is an elliptic operator of order $4$ (in the $x$-variable),  (\ref{smoothnes_and_reg_po_t}) and elliptic regularity implies
 $u_{\eps}\in H^{\infty}((0,T),H^6((0,1)))$. Using this improved regularity property in turn  we obtain by a
successive application of elliptic regularity
 \begin{equation*}
  u_{\eps}\in H^{\infty}(X_T)
   \subseteq C^{\infty}(\overline{X_T}).
 \end{equation*}
Thus, requirement 1.) above is proved.

It remains to show the moderateness property 2.). Due to \eqref{T_all_times} the statement is already verified for  derivatives of arbitrary orders with respect to time and up to order $2$ with respect to $x$.

\noindent\emph{Step 2:} From the results of Step 1 and equation \eqref{Peps} we deduce that
$$
  h_\eps := \d_x^2(c_\eps\, \d_x^2 u_\eps) = g_\eps -
     b_\eps\, \d_x^2 u_\eps - \d_t^2 u_\eps
$$
satisfies for all $l \in \N$ with some $N_l$ an estimate
\begin{equation}\label{mod22}
 \| \d_t^l  h_\eps \|_{L^2(X_T)}
    = O(\eps^{-N_l}) \quad (\eps\to 0).
 \end{equation}

Integrating the equation $0 = \d_x^2(c_{\eps}\d_x^2 u_{\eps}) - h_{\eps}$ twice with respect to $x$ gives
 \begin{equation}\label{mod2}
  e_{\eps}(t) + d_{\eps}(t) \, x =
     c_{\eps}(x) \, \d^2_xu_{\eps}(x,t) -
     \int_0^x\int_0^y h_{\eps}(z,t)dz dy.
 \end{equation}

Evaluating this equation at $x=0$ and $x = 1$
we obtain  $e_{\eps}(t) = c_\eps(0)\, \d_x^2 u_\eps(0,t)$ and $d_{\eps}(t) =  c_\eps(1)\, \d_x^2 u_\eps(1,t) - \int_0^1\int_0^y h_{\eps}(z,t)dz dy$. Thus \eqref{T_all_times} and \eqref{mod22} show that $(e_\eps)_\eps$ and $(d_\eps)_\eps$ belong to $\E_{M,H^\infty((0,1))}$.

Equation (\ref{mod2}) yields the following formula
for all $l,k \in \N, k \geq 2$:
\beq \label{h_rep}
 \d_t^l \d^{k}_x u_{\eps}(x,t) = \d_x^{k-2}
 \left(\frac{\d_t^l e_\eps(t) + \d_t^l d_{\eps}(t)\, x}{ c_{\eps}(x)} +
   \frac{1}{ c_{\eps}(x)}
  \int_0^x\int_0^y \d_t^l h_{\eps}(z,t)dz dy\right).
\eeq
Using the fact that $\frac{1}{ c_{\eps}}$
 is moderate (by boundedness away from $0$) we immediately obtain from \eqref{mod22} and \eqref{h_rep}  the moderateness estimates for $\|\d^l_t\d^k_xu_{\eps}\|_{L^2(X_T)}$ with $k=3,4$ as well ($l$ arbitrary), thus \eqref{T_lk} holds for all $l\in\N$ and $k=0,\ldots,4$.

Applying now \eqref{T_lk} with $k$ up to $4$  in (appropriate derivatives of) the defining equation for $h_\eps$ above we obtain the following improvement of \eqref{mod22}
$$
\forall l\in\N, 0 \leq p \leq 2,\ \exists N_{lp}: \quad
 \| \d_x^p \d_t^l  h_\eps \|_{L^2(X_T)}
    = O(\eps^{-N_{lp}}) \quad (\eps\to 0).
$$
Using this in turn in \eqref{h_rep} implies \eqref{T_lk} for $k$ up to order $6$. Successively proceeding in this way, we conclude that \eqref{T_lk} holds for all $l,k \in \N$. Thus property 2.) is shown and therefore $u = [(u_\eps)_\eps]$ defines a solution $u \in \G_{H^\infty(X_T)}$ to the initial value problem.

By construction we have $u_\eps(t) \in H^2_0((0,1))$ for all $t \in [0,T]$, hence $u(0,t) = u(1,t) = 0$ and $\d_x u(0,t) = \d_x u(1,t) = 0$ and $u$ satisfies the boundary conditions too.

{\bf Uniqueness:}\enspace \enspace  Assume that
$u = [(u_\eps)_\eps]$ satisfies initial-boundary value problem with zero initial values and right-hand side. Then we have for all $q\geq 0$
$$
   \|f_{1,\eps}\| = O(\eps^q),
   \|f_{2,\eps}\| = O(\eps^q),
   \|g_{\eps}\|_{L^2(X_T)} = O(\eps^q) \quad (\eps\to 0)
$$
and the energy estimate (\ref{EEeps1}) together with (\ref{const_D_eps}-\ref{const_F_eps}) imply for all $q \geq 0$ an estimate
$$
    \|u_\eps\|_{L^2(X_T)} = O(\eps^q) \quad (\eps \to 0).
$$
Therefore $(u_\eps)_\eps \in  {\cal N}_{H^{\infty}(X_T)}$ which proves that $u = 0$.
 \end{proof}

 \begin{remark}\label{after_main_thm}
 \begin{enumerate}
\item The above proof is valid without change if we replace  condition \eqref{Col_stiff} on $c \in \G_{H^\infty((0,1))}$ by the considerably
 weaker condition that $c$ is \emph{strtictly positive}, i.e.\ there exists $q \geq 0$ and $\eps_0 > 0$ such that
\beq \tag{\ref{Col_stiff}'}
  c_\eps \geq \eps^q \qquad (0 < \eps < \eps_0)
\eeq
(and no extra boundedness from above [other than the standard  moderateness]), since
this clearly implies that $1/c$ is a generalized function (cf.\ also \cite[Theorem 1.2.5]{GKOS:01}). Therefore, the existence and uniqueness result is extended to cases with stiffness that may be infinitely close to $0$ and need not be bounded from above.

\item In Step 2 of the proof we took advantage of having only   one spatial dimension. However, for a variant in higher space dimensions basic conclusions similar to those we drew from  formula \eqref{h_rep} could be reached by pseudo-differential techniques as in \cite{GGO:05} with the help of a Colombeau generalized parametrix of the operator $v \mapsto \d_x^2(c\, \d_x^2 v)$.

\item We point out that along the way in the proof of Theorem \ref{main_thm} (namely in Step 1) we have proved the following \emph{regularity result} for the weak solution:

If in addition to the hypotheses of Theorem \ref{ThmClassic} the coefficients $b$ and $c$ as well the data $f_1$, $f_2$, and $g$ are smooth, then the unique solution $u$ is also smooth, more precisely $u_{\eps}\in H^{\infty}(X_T)
   \subseteq C^{\infty}(\overline{X_T})$.

\end{enumerate}
 \end{remark}

\subsection{Regularity and coherence with smooth and weak solutions}

As a first simple comparison results displaying an enjoyable consistent relationship between classical smooth solutions, weak solutions, and generalized solutions  we consider the case of $C^\infty$ coefficients.

\begin{proposition} Assume in addition to the hypotheses of Theorem \ref{main_thm} that the coefficients $b$ and $c$ are (Colombeau classes of) $C^\infty$ functions.
\begin{enumerate}
\item If the data $f_1$, $f_2$, and $g$ are (the Colombeau classes of) smooth functions as well, then the unique generalized solution $u \in  \G_{H^\infty(X_T)}$ to the initital-boundary value problem \eqref{P} equals the (embedding of the) unique smooth solution in $H^{\infty}(X_T)$.

\item Let $f_1 = [(f_{1,\eps})_\eps]$, $f_2 = [(f_{2,\eps})_\eps]$, and $g = [(g_{\eps})_\eps]$ and assume that we have the following convergence properties of regularized data as $\eps \to 0$:  $f_{1,\eps} \to \tilde{f_1}$ in $H^2_0((0,1))$, $f_{2,\eps} \to \tilde{f_2}$ in $L^2((0,1))$, and $g_{\eps} \to \tilde{g}$ in $L^2(X_T)$. Then the unique generalized solution $u \in  \G_{H^\infty(X_T)}$ is associated with the unique weak solution $w \in L^2((0,T),H^2_0((0,1)))$ according to Theorem \ref{ThmClassic} with initial values $\tilde{f_1}$, $\tilde{f_2}$, and right-hand side $\tilde{g}$.

\end{enumerate}
\end{proposition}
\begin{proof} (i) Put $u_\eps := v$ for all $\eps$, where $v$ denotes the unique classical smooth solution (smoothness of $v$ is due to Remark \ref{after_main_thm}(iii)). From uniqueness of the Colombeau solution $u$ we must have $u = [(v)_\eps]$.\\[1mm]
(ii) This assertion follows immediately from the abstract theorem \cite[Chapter XVIII, Section 5, Subsection 4.2, Theorem 2]{DL:V5} on continuity of the solution with respect to the initial values and the right-hand side.
\end{proof}

Finally, we investigate intrinsic regularity of the Colombeau generalized solution $u$ in the sense of \eqref{Col_reg}. i.e.\ in terms of uniform $\eps$-asymptotics for all derivatives. In the context of  regularity theory for (pseudo-) differential equations a key property required of the coefficients (or symbols) has been shown to involve the notion of so-called \emph{slow scale nets} (\cite{GGO:05,HO:04,HOP:05}): A net $(r_\eps)_{\eps \in ]0,1]}$ of complex numbers is said to be of slow scale, if
$$
  \forall p \in \N: \quad
     |r_\eps|^p = O(\eps^{-1}) \quad (\eps \to 0).
$$
Note that any product of finitely many slow scale nets is of slow scale.
A generalized function $v = [(v_\eps)_\eps] \in \G_{H^\infty}$ is said to be of slow scale in all derivatives, if the net $(\|\d^\alpha  v_\eps\|_{L^2})_\eps$ is of slow scale for all $\alpha$.

Assuming slow scale conditions on the coefficients we show that regularity of the initial values and the right-hand side is preserved in the solution. This is a generalization of the regularity result for weak solutions described in Remark \ref{after_main_thm}(iii).

\begin{theorem} Let $b\in {\cal G}_{H^{\infty}(X_T)}$  and $c\in {\cal G}_{H^{\infty}(0,1))}$ be of slow scale in all derivatives. Assume further that $c$ satisfies \eqref{Col_stiff} and that $b$ is of slow scale $L^\infty$-log-type, i.e.\  for some (hence any) representative $(b_\eps)_\eps$ of $b$ there exist a slow scale net $(r_\eps)_\eps$ with $r_\eps \geq 1$ and $\eps_0 \in\, ]0,1]$ such that
\beq \label{sc_log_type} \tag{\ref{log_type}'}
  \| b_\eps \|_{L^\infty(X_T)} \leq
     \log(r_{\eps})
      \qquad (0 < \eps \leq \eps_0).
\eeq
If $f_1$ is as in Theorem \ref{main_thm} and in addition $f_1, f_2 \in \G^\infty_{H^\infty((0,1))}$, and $g\in \G^\infty_{H^\infty(X_T)}$, then the unique Colombeau solution $u$ to the initial-boundary value problem \eqref{P} belongs to $\G^\infty_{H^\infty(X_T)}$.
\end{theorem}

\begin{proof} From the hypotheses we may assume the following:
\begin{itemize}
\item[(A)] There exists $M_0 \in \N$ such that for all $k,l \in \N$
$$
   \|\d_x^k f_{1,\eps}\|^2_{L^2((0,1))} +
   \|\d_x^k f_{2,\eps}\|^2_{L^2((0,1))}
    + \| \d_t^l \d_x^k  g_{\eps} \|^2_{L^2(X_T)}
    = O(\eps^{-M_0}) \quad (\eps \to 0).
$$
\item[(B)] For all $k,l \in \N$ there is a slow scale net $(s_\eps)_\eps$, $s_\eps > 0$, such that
$$
  \|\d_x^k c_{\eps}\|_{L^\infty((0,1))} +
   \| \d_t^l \d_x^k  b_{\eps}\|_{L^\infty(X_T)}
   = O(s_\eps) \quad (\eps \to 0).
$$

\item[(C)]  The constants $D_T^\eps$ and $F_T^\eps$ in the  energy estimate \eqref{EEeps1} satisfy for small $\eps >0$ (with the slow scale net $(r_\eps)_\eps$ as in \eqref{sc_log_type})
\begin{align*}
  D_T^\eps &=
     O(\|c_\eps\|_{L^{\infty}}) = O(1) \\
  \exp(F_T^\eps) &=
   O(\exp(\|b_\eps\|_{L^{\infty}}))
  = O(r_\eps).
\end{align*}
\end{itemize}

Thus the basic energy inequality \eqref{EEeps1} immediately yields (with $M_0$ from property (A) above)

\beq \tag{$R_{l,k}$} \label{R_lk} \qquad
  \|\d^l_t\d^k_xu_{\eps}\|_{L^2(X_T)} =
  O(r_\eps\, \eps^{-M_0})
  = O(\eps^{-M_0-1})
  \qquad (\eps \to 0)
\eeq
for $(l,k) \in \{ (0,0),(1,0),(0,1),(0,2)\}$.

We proceed along the lines of Steps 1 and 2 in the proof of Theorem \ref{main_thm} and use the notation introduced there.

Thanks to (A) and (B) we obtain that $\tilde{f}_{1,\eps}$, $\tilde{f}_{2,\eps}$, and $\tilde{g}_{\eps}$ satisfy the following variant of (A):
\begin{itemize}
\item[(A')] For all $k,l \in \N$ there is a slow scale net $(s_\eps)_\eps$, $s_\eps > 0$, such that
$$
   \|\d_x^k \tilde{f}_{1,\eps}\|^2_{L^2((0,1))} +
   \|\d_x^k \tilde{f}_{2,\eps}\|^2_{L^2((0,1))}
    + \| \d_t^l \d_x^k  \tilde{g}_{\eps} \|^2_{L^2(X_T)}
    = O(s_\eps\cdot\eps^{-M_0}) \qquad (\eps \to 0).
$$
\end{itemize}
When $l=1$ and $k = 0,1,2$, or $l = 2$ and $k=0$, the energy estimate now implies
$$
  \|\d^l_t\d^k_xu_{\eps}\|_{L^2(X_T)} =
  O(r_\eps\, s_\eps\, \eps^{-M_0})
  = O(\eps^{-M_0-1})
  \qquad (\eps \to 0),
$$
which yields a first extension of \eqref{R_lk} to the cases $(l,k) = (1,1),(1,2),(2,0)$. The same kind of observation can be repeated for the next iteration and checked explicitly in terms of $\tilde{\tilde{f}}_{1,\eps}$, $\tilde{\tilde{f}}_{2,\eps}$, and $\tilde{\tilde{g}}_{\eps}$ to obtain the estimate \eqref{R_lk} for the pairs $(l,k) = (3,0),(2,1),(2,2)$ in addition.
Successively, we obtain that
\begin{center}\eqref{R_lk} holds for all $l \in \N$ and $k= 0,1,2$.
\end{center}

It is crucial to note that all the estimates \eqref{R_lk} established so far hold with stronger upper bounds of the form ``big oh of some slow scale net times $\eps^{-M_0}$'' and should be used in this form throughout the iterative process. It is only for extraction of regularity information a posteriori, where we deduce from these the (slightly weaker) uniform upper bound in the form $O(\eps^{-M_0-1})$. Keeping this in mind
an inspection of Step 2 in the proof of Theorem \ref{main_thm} enables us to reach the following conclusion:
Let $M_0$ be as in property (A) above, then for all $l,k \in \N$ the asymtotic estimate
$$  \qquad
  \|\d^l_t\d^k_xu_{\eps}\|_{L^2(X_T)} =
  O(\eps^{-M_0-1})
  \qquad (\eps \to 0)
$$
is valid. Therefore $u = [(u_\eps)_\eps]$ belongs to $\G^\infty_{H^\infty(X_T)}$, i.e.\ is a regular Colombeau generalized function.
\end{proof}

\paragraph{Acknowledgment:} Both authors express their hearty gratitude to Teodor Atanackovic, professor of mechanics at the University of Novi Sad, for providing the basics and helpful background information on the Euler-Bernoulli model and to Stevan Pilipovi\'c, professor of mathematics at the University of Novi Sad, for critical discussions and sharing the observation described in Remark \ref{after_main_thm}(i).

 \bibliographystyle{abbrv}

\bibliography{gue}

 \end{document}